\pdfoutput=1
\documentclass[10pt,twoside]{article}

\usepackage{lmodern}
\usepackage[T1]{fontenc}
\usepackage[utf8]{inputenc}
\usepackage{textcomp}
\usepackage[english]{babel}

\usepackage{amsmath}
\usepackage{amssymb}
\usepackage{amsthm}
\usepackage{mathtools}
\usepackage{bigstrut}
\usepackage{cases}
\usepackage{graphicx}
\usepackage{tabularx}
\usepackage{enumerate}
\usepackage{comment}

\usepackage{bm}

\usepackage[top=31mm, bottom=31mm, left=33mm, right=33mm]{geometry}

\usepackage{fancyhdr}
\DeclareMathOperator*{\nnz}{nnz}

\DeclareMathOperator*{\supp}{supp}
\DeclareMathOperator*{\mat}{mat}
\DeclareMathOperator*{\vecz}{vec}
\DeclareMathOperator*{\offdiag}{offdiag}
\DeclareMathOperator*{\argmin}{arg\,min}
\DeclareMathOperator*{\essinf}{ess\,inf}

\newtheorem*{thm}{Theorem}
\newtheorem*{cor}{Corollary}

\mathtoolsset{showonlyrefs,showmanualtags}

\usepackage[pdftex,colorlinks,pdftitle={Inverse parametric parabolic boundary value problem},pdfauthor={Lauri Mustonen}]{hyperref}
\hypersetup{linkcolor=black}
\hypersetup{urlcolor=black}
\hypersetup{citecolor=black}

\begin{document}

\title{Numerical study of a parametric parabolic equation and a~related inverse boundary value problem}
\author{Lauri Mustonen\footnote{Aalto University, Department of Mathematics and Systems Analysis, P.O.\ Box 11100, FI-00076 Aalto, Finland (lauri.mustonen@aalto.fi). This work was supported by the Academy of Finland (decision 267789) and the Finnish Foundation for Technology Promotion TES.}}

\maketitle

\begin{abstract}
We consider a time-dependent linear diffusion equation together with a related inverse boundary value problem.
The aim of the inverse problem is to determine, based on observations on the boundary, the non-homogeneous diffusion coefficient in the interior of an object.
The method in this paper relies on solving the forward problem for a whole family of diffusivities by using a spectral Galerkin method in the high-dimensional parameter domain.
The evaluation of the parametric solution and its derivatives is then completely independent of spatial and temporal discretizations.
In case of a quadratic approximation for the parameter dependence and a direct solver for linear least squares problems, we show that the evaluation of the parametric solution does not increase the complexity of any linearized subproblem arising from a Gauss--Newtonian method that is used to minimize a Tikhonov functional.
The feasibility of the proposed algorithm is demonstrated by diffusivity reconstructions in two and three spatial dimensions.
\end{abstract}

\section{Introduction}

Inverse boundary value problems arise in situations where one tries to find information about interior properties of an object by boundary measurements.
In electrical impedence tomography (EIT), for example, an electric current is injected into the body and the corresponding voltages are measured across the boundary \cite{borcea02}.
The aim is to reconstruct the electrical conductivity as a function, or to locate conductivity anomalies having prescribed (e.g., constant or vanishing) conductivity values.
The inverse problem is nonlinear and ill-posed, whereas the forward problem, namely, determining the boundary voltages when the conductivity and current patterns are given, is governed by a well-posed elliptic partial differential equation.
Mathematically equivalent applications include electrical capacitance tomography \cite{soleimani05}.

In this paper, we consider the inverse boundary value problem of a time-dependent diffusion equation.
This could be a model for thermal tomography, where the thermal diffusivity is to be reconstructed \cite{bakirov05}.
We assume that the heat flux at the boundary, as well as the initial temperature, are controlled and that the boundary temperatures are measured.
The diffusivity is assumed to be time-independent.
Unlike in the EIT problem, where the steady-state voltages are measured for several current patterns, in this paper we treat a single (or a few) initial and boundary conditions and measure the boundary temperatures at several instances of time.
Similar inverse boundary value problems for stationary inclusion-type diffusivities have been examined in, e.g., \cite{chapko98,chapko99,harbrecht13,ikehata10}.
In \cite{kolehmainen08}, both heat capacity and thermal conductivity were reconstructed simultaneously by using a least squares approach with several boundary conditions.
In \cite{toivanen12}, an extension to unknown surface flux was proposed and the feasibility of the method was verified with experimental three-dimensional data in \cite{toivanen14}.
The case of time-varying inclusions has been studied in \cite{gaitan15}.
For theoretical treatment of general inverse parabolic problems, we refer to~\cite{isakov06}.

Our approach to the inverse problem is based on writing the forward problem as a parametric differential equation.
To that end, the temperature is viewed as a function depending not only on spatial and temporal variables, but also on parameters that define the diffusivity function.
For the high-dimensional parameter domain, we adopt the spectral Galerkin method when the numerical solution is sought.
The spatial domain is discretized with the finite element method.
The combination was dubbed \emph{stochastic Galerkin finite element method} in, e.g., \cite{babuska04}, where the parameters were interpreted as random variables.
In that context, the spectral discretization is often called (generalized) \emph{polynomial chaos}.
Since \cite{ghanem91}, both Galerkin and collocation methods have been thoroughly studied and analyzed for several kinds of uncertainty quantification problems, including the time-dependent random diffusion equation \cite{back11,lemaitre10,nobile09,schwab11,xiu10,xiu03,xiu09}.

The time integration is performed by using an additive semi-implicit Euler method, a special case of so called implicit-explicit (IMEX) methods \cite{ascher95}.
In particular, in \cite{zhong96} these methods were proposed for stiff systems resulting from hypersonic transient reactive flows.
Although only first-order in time, in \cite{xiu09} the method was shown to be unconditionally stable for various discretizations of parametric and stochastic diffusion equations.
We demonstrate that when using locally supported functions to represent the diffusivity, the semi-implicit method is memory-optimal in the sense that storing and solving the resulting linear system requires strictly less space than the solution itself.

After solving the parametric forward problem, obtaining the boundary values for different diffusivities is merely a task of polynomial evaluation.
Indeed, in this paper we show that for a second-order approximation of the parameter dependence, the evaluation costs of the parametric solution and its Jacobian matrix do not increase the overall complexity of the inverse problem, if a (regularized) least squares minimization scheme based on Gauss--Newton method and QR decomposition is used to find the solution.
These minimization schemes include the standard Gauss--Newton method as well as its trust region counterparts, which from computational point of view are very similar.
Higher-order approximations increase the computational burden, but the workload of the inverse problem is still completely independent of the spatial and temporal discretizations of the forward problem.
However, for large-scale nonlinear inverse problems it is arguably recommendable to resort to conjugate gradient iterations instead of QR decompositions when solving the linearized subproblems (see, e.g., \cite{hanke97}).
The analysis of combining polynomial approximations and iterative linear least squares solvers is left for future studies.

An algorithmically similar inversion approach, albeit with Bayesian paradigm, was recently carried out for time-independent EIT problem in \cite{hakula14}.
To our knowledge, this is the first time when a parametric spectral solution to the forward problem is utilized for solving an inverse parabolic boundary value problem.
The numerical results show that the method is capable of reconstructing the diffusion coefficient based on boundary measurements.
Naturally, the quality of the reconstruction depends on the noise level of the measurement.
In addition, the method is more suitable for smooth diffusivities.
The distinctive feature of the approach is that once the parametric solution is saved, it takes only a few seconds to create the reconstruction from a given set of measurements.
Indeed, the obvious advantage of the method is that spatial and temporal discretizations of the forward problem do not affect the complexity of reconstructing the diffusivity.

This paper is organized as follows.
In section~\ref{sec:forward}, the precise mathematical formulation of the model problem is given.
Then we re-formulate the problem in a parametric sense and discretize the equation in spatial, parametric and temporal dimensions.
Section~\ref{sec:inverse} considers the inverse problem from the computational point of view.
Numerical examples in two and three spatial dimensions are provided in section~\ref{sec:num} and some conclusions are drawn in section~\ref{sec:conclusions}.

\section{Parametric forward model}
\label{sec:forward}

\subsection{Problem setting}
\label{ssec:setting}

Let $\varOmega \subset \mathbf{R}^d$, $d \geq 2$, be a bounded domain with Lipschitz boundary $\partial \varOmega$ and exterior unit normal $\hat{\boldsymbol{n}} \in [L^\infty(\partial \varOmega)]^d$. Furthermore, let $T>0$ be given.
The parabolic initial/boundary value problem considered in this paper is to find $u \colon \varOmega \times [0,T] \to \mathbf{R}$ satisfying
  \begin{equation}
  \label{eq:model}
    \begin{dcases}
    \partial_t u - \nabla \cdot (a \nabla u) = f & \text{in } \varOmega \times (0,T), \\
    a \nabla u \cdot \hat{\boldsymbol{n}} = g & \text{on } \partial \varOmega \times (0,T), \\
    u = u_0 & \text{in } \varOmega \times \{0\},
    \end{dcases}
  \end{equation}
where $f \in L^2((0,T);L^2(\varOmega))$ and $g \in L^2((0,T);H^{-1/2}(\partial \varOmega))$ denote the interior source and the boundary flux, respectively, and $u_0 \in L^2(\varOmega)$ is the initial condition.
The diffusivity (or diffusion coefficient) $a$ belongs to $L_+^\infty(\varOmega)$, where
  \begin{equation}
  \label{eq:Linftyplus}
  L_+^\infty(\varOmega) \coloneqq \{ a \in L^\infty(\varOmega) \mid \essinf a > 0 \}.
  \end{equation}
In particular, the diffusivity is assumed to be independent of the time variable $t$.
The variational formulation of the problem \eqref{eq:model} is to find $u$ such that
  \begin{equation}
  \label{eq:variational}
  (\partial_t u, v)_{L^2(\varOmega)} + (a \nabla u, \nabla v)_{[L^2(\varOmega)]^d} = (f,v)_{L^2(\varOmega)} + \langle g,\gamma v \rangle,
  \end{equation}
accompanied with an appropriate initial condition, holds for all $v \in H^1(\varOmega)$ and for almost every $t \in (0,T)$. Here, $\gamma \colon H^1(\varOmega) \to H^{1/2}(\partial \varOmega)$ denotes the trace operator and $\langle \cdot, \cdot \rangle$ is the duality pairing between $H^{-1/2}(\partial \varOmega)$ and $H^{1/2}(\partial \varOmega)$.
We refer to solving $u$ from \eqref{eq:model} or \eqref{eq:variational} as the \emph{forward problem}. 
For the variational form \eqref{eq:variational} there exists a unique solution $u$ satisfying
  \begin{equation}
  \label{eq:solution}
  u \in L^2((0,T); H^1(\varOmega)) \cap C^0([0,T]; L^2(\varOmega))
  \end{equation}
(see, e.g., \cite[Chap.~XVIII]{dautray92}).

By \emph{inverse problem} we mean finding a diffusivity $a \in L_+^\infty(\varOmega)$ such that the solution $u$ of the corresponding forward problem matches with given data.
More precisely, we consider some given boundary data $\tilde{U} \colon \partial \varOmega \times (0,T) \to \mathbf{R}$ and the corresponding trace $U(a) \coloneqq \gamma u \in L^2((0,T); H^{1/2}(\partial \varOmega))$ of the forward solution.
It is not obvious whether the equation $U(a) = \tilde{U}$ admits a (unique) solution~$a$; see \cite[Chap.~9]{isakov06} for some related results with slightly different assumptions.
In any case, the inverse problem is ill-posed in the sense that the solution (if it exists) does not depend continuously on the data $\tilde{U}$ in any reasonable metric.

A physical interpretation of the problem \eqref{eq:model} could be that $u$ represents the temperature and $a$ is the thermal diffusivity in an object $\varOmega$ which has unit heat capacity.
The setting is easily extended to a non-homogeneous but \emph{known} heat capacity, in which case $a$ denotes the thermal conductivity \cite{bakirov05}.
Although we restrict ourselves to the model problem \eqref{eq:model}, the presented methods are, with minor changes, widely applicable to other types of problems.
For example, different boundary conditions, including Dirichlet, Robin and mixed, can be handled easily.
Many of the observations apply to elliptic parametric problems as well.

\subsection{Parametrization of the diffusivity}
In order to treat the inverse problem $U(a) = \tilde{U}$ numerically, we assume that the diffusivity is characterized by a finite number of parameters and an injective mapping $a \colon \varTheta \to L_+^\infty(\varOmega)$, where $\varTheta \subseteq \mathbf{R}^P$ is a high-dimensional parameter domain, is given.
Besides this, our aim in this section is to establish an explicit parameter dependence for the solution $u$ as well.
That is, we seek a numerical solution to the forward problem in the form $u \colon \varOmega \times [0,T] \times \varTheta \to \mathbf{R}$.
The existence and uniqueness (similar to \eqref{eq:solution}) of such solution is obvious.

The most general methods for obtaining the parametric forward solution are based on solving the regular problem \eqref{eq:model} or \eqref{eq:variational} for a set of \emph{collocation points} in the parameter domain.
The full parametric solution would then be written by either using the collocation points as quadrature nodes and projecting the solution to a chosen basis by numerical integration, or by interpolating the solution to the whole domain $\varTheta$ by using a high-dimensional interpolation rule.
In both cases, the number of collocation points needed is typically very high and the construction of the points is often based on sparse Smolyak grids \cite{babuska10,nobile08}.
A related strategy is to use Galerkin method with double orthogonal polynomials, which also results in a decoupled system of equations \cite{frauenfelder05}.

In this paper, we discretize the parameter domain by the spectral Galerkin approach that yields a large coupled system.
We introduce a positive weight function $w \in L^1(\varTheta)$ and employ the corresponding weighted $L^2$ spaces.
The variational form of the forward problem then becomes to find $u \colon \varOmega \times [0,T] \times \varTheta \to \mathbf{R} $ such that for all $v \in L_w^2(\varTheta; H^1(\varOmega))$ and a.e.\ in $(0,T)$
  \begin{equation}
  \label{eq:paramVariational}
  (\partial_t u, v)_{L_w^2(\varTheta ; L^2(\varOmega))} + (a \nabla u, \nabla v)_{L_w^2(\varTheta; [L^2(\varOmega)]^d)}
  = (f,v)_{L_w^2(\varTheta ; L^2(\varOmega))} + ( \langle g,\gamma v \rangle ),
  \end{equation}
where $( \langle \cdot, \cdot \rangle )$ denotes the duality pairing between the weighted spaces $L_w^2(\varTheta; H^{-1/2}(\partial \varOmega))$ and $L_w^2(\varTheta; H^{1/2}(\partial \varOmega))$.
Here, the functions $f$ and $g$ are assumed to be constant with respect to the parameters.
We \emph{assume} that there exists a unique solution
  \begin{equation}
  u \in L_w^2 \big(\varTheta; L^2 ( (0,T); H^1(\varOmega) ) \cap C^0 ([0,T]; L^2(\varOmega) ) \big).
  \end{equation}
This is guaranteed, for example, if $\varTheta$ is bounded and there exist positive constants $a_\text{min}$ and $a_\text{max}$ such that
  \begin{equation}
  a_\text{min} \leq a(\boldsymbol{x}, \boldsymbol{\vartheta}) \leq a_\text{max}
  \end{equation}
for almost every $(\boldsymbol{x}, \boldsymbol{\vartheta}) \in \varOmega \times \varTheta$.

A convenient way to define the parametrized diffusivity $a \colon \varTheta \to L_+^\infty(\varOmega)$
is to write
  \begin{equation}
  \label{eq:a}
  a(\boldsymbol{\vartheta}) = \sum_{p=1}^P \vartheta_p \psi_p
  \end{equation}
for $\boldsymbol{\vartheta} \in \varTheta$, or slightly more generally
  \begin{equation}
  \label{eq:aMean}
  a(\boldsymbol{\vartheta}) = \bar{a} + \sum_{p=1}^P \vartheta_p \psi_p
  \end{equation}
for some ``background diffusivity'' $\bar{a} \in L_+^\infty(\varOmega)$.
In this paper, we resort to \eqref{eq:a} and choose $\{\psi_p\}_{p=1}^P \subset L^\infty(\varOmega)$ to be nonnegative and such that they form a partition of unity.
Hence, the resulting diffusivity $a$ is bounded by the extremal values of the parameters, i.e.,
  \begin{equation}
  \inf_{\boldsymbol{\vartheta} \in \varTheta} \min_{1 \leq p \leq P} \vartheta_p \leq a(\boldsymbol{\varsigma}) \leq \sup_{\boldsymbol{\vartheta} \in \varTheta} \max_{1 \leq p \leq P} \vartheta_p \qquad \text{in } \varOmega
  \end{equation}
for all $\boldsymbol{\varsigma} \in \varTheta$.
Naturally, the positivity requirement \eqref{eq:Linftyplus} imposes restrictions on choosing the parameter domain $\varTheta$ once the functions $\psi_p$ are fixed.

One possible family $\{\psi_p\}_{p=1}^P$ is a set of B-splines, which for a rectangular domain $\varOmega$ are easy to construct as tensor products of univariate splines.
Let us briefly recap the basic properties of B-splines, following \cite{hollig03}.
A standard univariate uniform B-spline of degree $s \in \mathbf{N}$ can be defined recursively as a convolution
  \begin{equation}
  b_s(x) \coloneqq \int_{-\infty}^\infty b_{s-1}(x-y) b_0(y) \,\mathrm{d}y,
  \end{equation}
where $b_0(x)$ is the indicator function of the interval $[0,1)$.
We see that $b_s \in C^{s-1}(\mathbf{R})$ and $\supp(b_s) = [0, s+1]$.
The splines $b_s(x+s), \ldots, b_s(x)$ form a partition of unity on the interval $[0,1)$.
The multivariate splines can be constructed as
  \begin{equation}
  b_{s,d}(\boldsymbol{x}) \coloneqq \prod_{i=1}^d b_s(x_i)
  \end{equation}
and again $b_{s,d} \in C^{s-1}(\mathbf{R}^d)$.
Transforming the standard splines to form a desired partition of unity is elementary.
See \cite{hollig03} for constructing splines for non-rectangular domains.
For polygonal and polyhedral domains one can also use piecewise linear functions $\psi_p$ such as those used in finite element solvers.

It can be shown that if $u$ and $\tilde{u}$ are two forward solutions corresponding to arbitrary diffusivities $a$ and $\tilde{a}$, respectively (in $L_+^\infty(\varOmega)$, but not necessarily of the form \eqref{eq:a}), then there exists $C_1>0$ such that
  \begin{equation}
  \lVert u - \tilde{u} \rVert_{L^2((0,T); L^2(\varOmega))} \leq C_1 \lVert a - \tilde{a} \rVert_{L^\infty(\varOmega)},
  \end{equation}
see, e.g., \cite{hoang13}.
On the other hand, the approximation error result for tensor product B-splines $\{\psi_p\}_{p=1}^P$ states that if $S$ contains all the functions of the form \eqref{eq:a} with $\varTheta = \mathbf{R}^P$, then
  \begin{equation}
  \inf_{a \in S} \lVert a - \tilde{a} \rVert_{L^\infty(\varOmega)} \leq C_2 \kappa^{s+1} \max_{1 \leq i \leq d} \lVert \partial_i^{s+1} \tilde{a} \rVert_{L^\infty(\varOmega)}
  \end{equation}
for any diffusivity $\tilde{a} \in L^\infty(\varOmega)$ and spline degree $s \geq 0$, where $C_2(s)>0$ and $\kappa \sim P^{-1/d}$ is a characteristic distance between spline knots \cite{dahmen80}.

Substituting \eqref{eq:a} into \eqref{eq:paramVariational} leads to
  \begin{equation}
  \label{eq:paramVarParam}
  (\partial_t u, v) + \sum_{p=1}^P (\iota_p \psi_p \nabla u, \nabla v) = (f,v) + ( \langle g,\gamma v \rangle ),
  \end{equation}
where we have dropped the subscripts for brevity, and where the projection $\iota_p \colon \varTheta \to \mathbf{R}$ is defined by $\iota_p(\boldsymbol{\vartheta}) = \vartheta_p$.
In the following, we discretize the spaces $L_w^2(\varTheta)$ and $H^1(\varOmega)$ in order to recast \eqref{eq:paramVarParam} as a matrix equation, which is then solved by performing a finite difference discretization of the temporal domain.

For a separable Hilbert space $H$, the isomorphism $L_w^2(\varTheta; H) \simeq L_w^2(\varTheta) \otimes H$ holds, if $L_w^2(\varTheta)$ is also separable (see, e.g., \cite[Sec.~B.3]{schwab11} and references therein).
Now if $\{\phi_i\}$ and $\{\varphi_j\}$ are countable bases for $H^1(\varOmega)$ and $L_w^2(\varTheta)$, respectively, then $\{\phi_i \varphi_j\}$ forms a basis for $L_w^2(\varTheta; H^1(\varOmega))$.
This suggests looking for a numerical solution to the parametric forward \eqref{eq:paramVarParam} problem in the form
  \begin{equation}
  u_{M,N}(\boldsymbol{x},t,\boldsymbol{\vartheta}) = \sum_{i=1}^M \sum_{j=1}^N \hat{u}_{i,j}(t) \phi_i(\boldsymbol{x}) \varphi_j(\boldsymbol{\vartheta})
  \end{equation}
(where the finite-dimensional bases are not necessarily subsets of the aforementioned countable bases) with the time-independent test functions $v \in L_w^2(\varTheta; H^1(\varOmega))$ having the form $v_{i,j} = \phi_i \varphi_j$.

The semi-discrete equation for the parametric forward problem \eqref{eq:paramVarParam} can now be written as
  \begin{equation}
  \label{eq:ode}
  \partial_t \boldsymbol{B} \hat{\boldsymbol{u}}(t) + \boldsymbol{A} \hat{\boldsymbol{u}}(t) = \hat{\boldsymbol{r}}(t),
  \end{equation}
where $\hat{\boldsymbol{u}} \colon (0,T) \to \mathbf{R}^{MN}$ denotes the vector of unknown coefficients
  \begin{equation}
  \label{eq:uhat}
  \hat{\boldsymbol{u}}(t) \coloneqq [\hat{u}_{1,1}(t), \hat{u}_{2,1}(t), \ldots, \hat{u}_{M,1}(t), \hat{u}_{1,2}(t), \ldots, \hat{u}_{M-1,N}(t), \hat{u}_{M,N}(t)]^\mathrm{T},
  \end{equation}
$\boldsymbol{B} \in \mathbf{R}^{MN \times MN}$ is the parametric mass matrix defined by its entries
  \begin{equation}
  \label{eq:Rijkl}
  B_{i,j,k,l} = (\varphi_j,\varphi_l)_{L_w^2(\varTheta)} (\phi_i,\phi_k)_{L^2(\varOmega)}
  \end{equation}
and $\boldsymbol{A} \in \mathbf{R}^{MN \times MN}$ is the parametric stiffness matrix satisfying
  \begin{equation}
  \label{eq:Aijkl}
  A_{i,j,k,l} = \sum_{p=1}^P (\iota_p \varphi_j, \varphi_l)_{L_w^2(\varTheta)} (\psi_p \nabla \phi_i, \nabla \phi_k)_{L^2(\varOmega)}.
  \end{equation}
The vector $\hat{\boldsymbol{r}} \colon (0,T) \to \mathbf{R}^{MN}$ on the right-hand side is
  \begin{equation}
  \label{eq:r}
  \hat{r}_{k,l}(t) = (1,\varphi_l)_{L_w^2(\varTheta)} \big((f(t),\phi_k)_{L^2(\varOmega)} + \langle g(t), \phi_k \rangle\big).
  \end{equation}

\subsection{Finite element discretization}
We construct a finite-dimensional subspace $V_h \subset H^1(\varOmega)$ by following the standard finite element method (FEM) procedures \cite{larsson03}.
The finite element basis functions are denoted by $\{ \phi_i \}_{i=1}^M$ and $h>0$ is the mesh size parameter of a regular enough mesh.
The contribution of the spatial parts in \eqref{eq:Rijkl} can be thought of as the symmetric positive-definite mass matrix $\boldsymbol{B}^\bullet \in \mathbf{R}^{M \times M}$ defined by
  \begin{equation}
  \label{eq:femMass}
  B^\bullet_{i,k} = (\phi_i, \phi_k)_{L^2(\varOmega)}.
  \end{equation}
For a non-homogeneous heat capacity, the elements of the mass matrix should be modified accordingly.

The spatial term in \eqref{eq:Aijkl} involves symmetric positive-semidefinite matrices of the stiffness type
  \begin{equation}
  A^{(p)}_{i,k} = (\psi_p \nabla \phi_i, \nabla \phi_k)_{L^2(\varOmega)}.
  \end{equation}
Note that the individual matrices $\boldsymbol{A}^{(p)}$ may be very sparse and rank-deficient if the functions $\psi_p$ are locally supported.
Let us denote by $\nnz(\boldsymbol{V})$ the number of nonzero elements of an arbitrary matrix $\boldsymbol{V}$ and introduce a sparsity quantity
  \begin{equation}
  \label{eq:eta}
  \eta \coloneqq \frac{\sum_{p=1}^P \nnz(\boldsymbol{A}^{(p)})}{\nnz(\boldsymbol{A}^\bullet)},
  \end{equation}
where $\boldsymbol{A}^\bullet \in \mathbf{R}^{M \times M}$ is the standard stiffness matrix defined by $A^\bullet_{i,k} = (\nabla \phi_i, \nabla \phi_k)$.
Clearly, $1 \leq \eta \leq P$.
Note also that 
  \begin{equation}
  \label{eq:pou}
  \sum_{p=1}^P \boldsymbol{A}^{(p)} = \boldsymbol{A}^\bullet.
  \end{equation}
due to the partition of unity property of $\{\psi_p\}_{p=1}^P$.

We do not consider the convergence with respect to $h$ in detail.
However, recall that the familiar a priori error estimate
  \begin{equation}
  \lVert u - u_h \rVert_{H^1(\varOmega)} \leq C h^{r-1} \lVert u \rVert_{H^{r}(\varOmega)}
  \end{equation}
for the numerical solution $u_h$ of a canonical elliptic problem depends on $r \geq 1$, for which
  \begin{equation}
  r \leq \sup_{q \in \mathbf{R}}\{q \mid u \in H^q(\varOmega)\},
  \end{equation}
and for which the employed finite element type also sets an upper bound.
Since for a noncontinuous coefficient function (corresponding to zeroth order splines and $s=0$), the solution $u$ is not in $H^2(\varOmega)$, the convergence becomes sublinear and thus deteriorates even for piecewise linear FEM basis functions.
Similarly, higher order FEM basis functions should be used only if the spline degree is high enough.
In practice, however, the cardinality of the finite element space is much larger than the number of splines.
Thus, the solution may converge rapidly regardless of the nonsmoothness of the diffusivity.
The compatibility of the diffusivity representation and finite element discretization, in terms of forward computations, is discussed in \cite{mustonen14}.

\subsection{Spectral Galerkin method}
Next, we discretize the space $L_w^2(\varTheta)$ by introducing an $N$-dimensional subspace $W \subset L_w^2(\varTheta)$.
We adopt the spectral Galerkin method, and therefore choose basis functions $\{\varphi_j\}_{j=1}^N$ that are orthogonal in the $L_w^2$-sense.
For convenience, we also assume normality, thus
  \begin{equation}
  (\varphi_j, \varphi_l)_{L_w^2(\varTheta)} = \delta_{j,l},
  \end{equation}
where $\delta$ is the Kronecker delta.
It immediately follows that the matrix corresponding to the parametric part of \eqref{eq:Rijkl} is the identity matrix $\boldsymbol{I} \in \mathbf{R}^{N \times N}$ and therefore $\boldsymbol{B} = \boldsymbol{I} \otimes \boldsymbol{B}^\bullet$.

We only consider hyperrectangular parameter domains, and for simplicity assume that the domain is a hypercube $\varTheta = E^P$ for some interval $E \subset \mathbf{R}$.
The $P$-variate basis functions can then be easily constructed as tensor products of univariate polynomials $\{ \bar{\varphi}_r \}_{r=0}^n$, where $n \in \mathbf{N}_0$ is the maximum univariate degree.
More precisely, we define
  \begin{equation}
  \label{eq:varphij}
  \varphi_j(\boldsymbol{\vartheta}) \coloneqq \prod_{p=1}^P \bar{\varphi}_{\varLambda_{j,p}}(\vartheta_p),
  \end{equation}
for $j = 1, \ldots, N$.
In addition, the weight function is assumed to be separable in the sense that
  \begin{equation}
  w(\boldsymbol{\vartheta}) = \prod_{p=1}^P \bar{w}_p (\vartheta_p).
  \end{equation}
(Often we may even have $\bar{w}_p = \bar{w}$ for all $p=1, \ldots, P$.)
The univariate degrees $\varLambda_{j,p}$ can be stored in a matrix $\boldsymbol{\varLambda} \in \mathbf{N}_0^{N \times P}$ which has $N$ (yet unspecified) distinct rows.
For convenience, we assume that the first row contains only zeros, that is, $\varphi_1 \equiv \lVert w \rVert_{L^1(\varTheta)}^{-1/2}$ is the constant polynomial.

Due to orthogonality, all but the first block in the right hand side vector \eqref{eq:r} vanish.
The stiffness matrix \eqref{eq:Aijkl} can be written as
  \begin{equation}
  \boldsymbol{A} = \sum_{p=1}^P (\boldsymbol{Y}^{(p)} \otimes \boldsymbol{A}^{(p)}),
  \end{equation}
where the symmetric matrices of size $N \times N$ involving the actual parameters are defined by
  \begin{equation}
  \label{eq:Ypjl}
  Y^{(p)}_{j,l} = (\iota_p \varphi_j, \varphi_l)_{L_w^2(\varTheta)}
  \end{equation}
for $p = 1, \ldots, P$.

For a fixed univariate degree $n$, the largest possible degree matrix $\boldsymbol{\varLambda}$ contains $N = (n+1)^P$ rows, which is too much for most practical purposes.
A widely used alternative is to limit the row sum of the degree matrix to equal $n$.
This results in a so called total degree polynomial space, which is spanned by
  \begin{equation}
  \label{eq:tdspace}
  \{ \varphi_j \}_{j=1}^N = \mathopen{}\left\{ \prod_{p=1}^P \bar{\varphi}_{r_p}(\vartheta_p) \;\middle|\; \sum_{p=1}^P r_p \leq n \; \right\}\mathclose{}.
  \end{equation}
We collect some combinatorial results related to the total degree space in the following theorem.
Hereafter, a square matrix $\boldsymbol{V}$ whose diagonal entries are replaced by zeros, is denoted by $\offdiag(\boldsymbol{V})$.

\begin{thm}
\label{thm:comb}
Let $\boldsymbol{\varLambda} \in \mathbf{N}_0^{N \times P}$ be a degree matrix corresponding to the polynomial basis of a total degree space of $P$ variables and a total degree $n$, where $1 \leq n \ll P$.
Then
  \begin{enumerate}[(a)]
  \item \begin{equation}
        N = \binom{P+n}{n} = \frac{(P+n)!}{P!\,n!} = \frac{P^n}{n!} + \mathcal{O}(P^{n-1})
        \end{equation}
  \item \begin{equation} \label{eq:nnzLambda}
        \nnz(\boldsymbol{\varLambda}) = P \binom{P+n-1}{n-1} = \frac{Pn}{P+n} N = \frac{P^n}{(n-1)!} + \mathcal{O}(P^{n-1})
        \end{equation}
  \item
        \begin{equation}
        \nnz(\offdiag(\boldsymbol{Y}^{(p)})) = \frac{2}{P} \nnz(\boldsymbol{\varLambda}) = \frac{2P^{n-1}}{(n-1)!} + \mathcal{O}(P^{n-2})
        \end{equation}
        and the positions of the nonzero elements of $\offdiag(\boldsymbol{Y}^{(p)})$ are disjoint for $1 \leq p \leq P$.
  \end{enumerate}
\end{thm}
\begin{proof}
We skip the proof here but refer to, e.g., \cite{ernst10,ghanem91} for similar results with proofs.
\end{proof}

In what follows, we will use the result $\nnz(\boldsymbol{\varLambda}) = \mathcal{O}(P^n)$ from part (b) of the theorem.
This is, of course, quite elementary consequence of part (a), since the definition of the total degree space immediately yields $\nnz(\boldsymbol{\varLambda}) \leq nN$.
Note that the quantities in parts (a) and (b) of the theorem are symmetric in terms of $P$ and $n$ (as is part (c) after summing over $p$).
Thus, the asymptotic formulae for the case $n \longrightarrow \infty$ can be obtained by switching the roles of $P$ and $n$.

Some alternatives for the total degree spaces can be found in \cite{chkifa13}.
We also mention that due to the locality of the splines, it seems reasonable to ignore some cross-terms including variables that correspond to coefficients of splines being supported far away from each other.
This idea is not completely heuristic, see the asymptotic formula in elliptic case with small inclusions \cite[Chap.~5]{ammari07}.
We have also numerically observed this kind of behaviour.

Ordinarily, one uses Legendre polynomials and a constant weight function, but other choices are possible as well.
We refer to \cite{gautschi04} for discussion on orthogonal polynomials and \cite{CHQZ2} for more detailed analysis of different types of spectral approximations.
The convergence rate of spectral approximation depends on the smoothness of the forward solution with respect to the parameters.
The affine representations \eqref{eq:a} and \eqref{eq:aMean} are known to result in an analytic parameter dependence in $\varOmega \times (0,T)$, see \cite{hoang13,nobile09}.
On the other hand, when the number of parameters is large, it is only possible to use a low polynomial degree and thus the asymptotic convergence is not the main interest.
Finally, we mention that the rest of this paper is equally applicable to the case where the parameter dependence is approximated with a spectral collocation method.
In some cases, for example when the diffusivity parametrization is based on a boundary curve between two different diffusivity values, the collocation method may be significally more straightforward to formulate and implement.

\subsection{Time integration}
Let us continue by discussing how the semi-discrete equation \eqref{eq:ode} can be solved in time.
In principle, any time integration method can be used to solve the resulting system of ordinary differential equations.
However, diffusion equations typically require very small time steps if an explicit time integration method is used.
This applies to parametric equations as well, see, e.g., \cite{powell09} for some eigenvalue bounds.
Implicit methods, on the other hand, require the solution of a full system of size $MN$ having a nontrivial sparsity structure.

Explicit methods with larger stability regions for parabolic equations can be derived from the class of so called Runge--Kutta--Chebyshev methods \cite{verwer96}.
An alternative to explicit and implicit methods are semi-implicit or implicit-explicit (IMEX) methods \cite{ascher95,zhong96}.
Here we resort to a first-order semi-implicit Euler method, which in \cite{xiu09} was shown to be unconditionally stable for parametric diffusion equations with symmetric Jacobi weights (in particular, with a constant weight and Legendre polynomials).
We already know that the mass matrix $\boldsymbol{B}$ is block-diagonal and actually has constant blocks.
The feasibility of the semi-implicit method is based on the diagonal dominance of the matrices $\boldsymbol{Y}^{(p)}$ and the decomposition
  \begin{equation}
  \boldsymbol{A} = \mu \boldsymbol{I} \otimes \boldsymbol{A}^\bullet + \sum_{p=1}^P (\offdiag(\boldsymbol{Y}^{(p)}) \otimes \boldsymbol{A}^{(p)}) \eqqcolon \boldsymbol{D} + \boldsymbol{S},
  \end{equation}
where $\boldsymbol{D} \in \mathbf{R}^{MN \times MN}$ is a block-diagonal matrix with constant blocks and $\mu > 0$.
This decomposition follows somewhat directly if the diffusivity is parametrized as in \eqref{eq:aMean} with $\bar{a} \equiv \mu$, see also \cite{ernst10}.

Let us denote $\hat{\boldsymbol{u}}^{(k)} \coloneqq \hat{\boldsymbol{u}}(k \delta)$, where $k \in \mathbf{N}$ and $\delta > 0$ is a time step.
Furthermore, let
  \begin{equation}
  \hat{\boldsymbol{r}}^{(k+1/2)} \coloneqq \frac{1}{\delta} \int_{k \delta}^{(k+1)\delta} \hat{\boldsymbol{r}}(t) \,\mathrm{d}t
  \end{equation}
be the mean value of the right hand side vector \eqref{eq:r} over one time step.
Starting from an initial vector $\hat{\boldsymbol{u}}^{(0)}$, which similar to $\hat{\boldsymbol{r}}(t)$ contains in general only $M$ nonzero values, the semi-implicit Euler method presented in \cite{xiu09} can be written as
  \begin{equation}
  \label{eq:timeode}
  (\boldsymbol{B} + \delta \boldsymbol{D}) \hat{\boldsymbol{u}}^{(k+1)}
  = (\boldsymbol{B} - \delta \boldsymbol{S}) \hat{\boldsymbol{u}}^{(k)} + \delta \hat{\boldsymbol{r}}^{(k+1/2)}.
  \end{equation}

Let us define
  \begin{equation}
  \mat(\hat{\boldsymbol{u}}(t)) \coloneqq
    \begin{bmatrix}
    \hat{u}_{1,1}(t) & \cdots & \hat{u}_{1,N}(t) \\
    \vdots & \ddots & \vdots \\
    \hat{u}_{M,1}(t) & \cdots & \hat{u}_{M,N}(t)
    \end{bmatrix} \in \mathbf{R}^{M \times N}
  \end{equation}
as the matricization of the vector $\hat{\boldsymbol{u}}(t)$ defined in \eqref{eq:uhat}.
Correspondingly, we define the vectorization $\vecz(\mat(\hat{\boldsymbol{u}}(t))) \coloneqq \hat{\boldsymbol{u}}(t) \in \mathbf{R}^{MN}$.
Due to the block structure of the matrices $\boldsymbol{B}$ and $\boldsymbol{D}$, the algorithm \eqref{eq:timeode} can be efficiently implemented based on the following rules:
  \begin{enumerate}
  \item Compute $\boldsymbol{\varXi} = \boldsymbol{B}^\bullet \mat(\hat{\boldsymbol{u}}^{(k)})$.
  \item Compute $\boldsymbol{\xi} = \vecz(\boldsymbol{\varXi}) - \delta \boldsymbol{S} \hat{\boldsymbol{u}}^{(k)} + \delta \hat{\boldsymbol{r}}^{(k+1/2)}$.
  \item Solve $\boldsymbol{\varUpsilon} \in \mathbf{R}^{M \times N}$ from $(\boldsymbol{B}^\bullet + \delta \mu \boldsymbol{A}^\bullet) \boldsymbol{\varUpsilon} = \mat(\boldsymbol{\xi})$.
  \item Set $\hat{\boldsymbol{u}}^{(k+1)} = \vecz(\boldsymbol{\varUpsilon})$.
  \end{enumerate}
The solution to the system in the third step can be obtained by explicitly storing the inverse of the (time-independent) matrix $\boldsymbol{B}^\bullet + \delta \mu \boldsymbol{A}^\bullet$ or storing the Cholesky factor and performing triangular substitutions.
For large $M$, iterative methods such as the deflated conjugate gradient method may also be used \cite{saad00}.

The matrices $\boldsymbol{B}^\bullet$ and $\boldsymbol{A}^\bullet$ have $\mathcal{O}(M)$ nonzero elements and their sparsity structure resulting from standard FEM discretization is well-studied and can be exploited.
The only matrix of size $MN \times MN$ that has to be stored in the proposed algorithm is $\boldsymbol{S}$, which is very sparse.
Indeed, theorem \ref{thm:comb} and the definition \eqref{eq:eta} yield
  \begin{equation}
  \nnz(\boldsymbol{S}) = \sum_{p=1}^P \nnz(\offdiag(\boldsymbol{Y}^{(p)})) \nnz(\boldsymbol{A}^{(p)}) = \frac{2nN \eta \nnz(\boldsymbol{A}^\bullet)}{P+n}.
  \end{equation}
If $\eta = \mathcal{O}(P)$, then $\nnz(\boldsymbol{S}) = \mathcal{O}(MN)$.
However, if $P$ is large and the functions $\psi_p$ are splines or otherwise supported on a small region, then $\eta = \mathcal{O}(1)$ and the matrix $\boldsymbol{S}$ has essentially less nonzero elements than the vectors $\hat{\boldsymbol{u}}^{(k)}$ for $k>0$.

Naturally, the vectors $\hat{\boldsymbol{u}}^{(k)}$ can be discarded after computing the next solution $\hat{\boldsymbol{u}}^{(k+1)}$.
Thus, arbitrarily small time steps can be used without imposing massive memory requirements, and the fact that the method is only first-order, is not necessarily an issue.
As suggested in \cite{xiu09}, the proposed method can be modified to perform a Jacobi iteration for parametric or stochastic elliptic equations.
In addition, we note that preconditioned Krylov subspace methods for time-independent parametric or stochastic problems may employ structures that are similar to those presented here.
See, e.g., \cite{ullmann10} and the references therein for more information on that subject.

\section{Estimating parameters from boundary data}
\label{sec:inverse}

In this section, we consider the inverse problem of determining the diffusion coefficient from boundary measurements.
As presented in section \ref{ssec:setting}, the continuous formulation of the inverse problem is to find a diffusivity $a \in L_+^\infty(\varOmega)$ such that $\gamma u$, which is the trace of the solution corresponding to $a$, equals (or is close to) the measurement $\tilde{U} \colon \partial \varOmega \times (0,T) \to \mathbf{R}$.

A practical measurement contains only finitely many, say $Q$, values.
That is, we consider a measurement vector $\tilde{\boldsymbol{U}} \in \mathbf{R}^Q$ satisfying $\tilde{U}_q \approx u(\boldsymbol{x}^{(q)}, t^{(q)})$, where $u$ is the temperature and $\{(\boldsymbol{x}^{(q)}, t^{(q)})\}_{q=1}^Q \subset \partial \varOmega \times (0,T)$ defines the physical coordinates of the observations.
We denote the measurement only as an approximation of the temperature due to unavoidable errors in measurements and uncertainties in the problem setting.
The parametric numerical solution corresponding to the coordinates $\{(\boldsymbol{x}^{(q)}, t^{(q)})\}$ can be written as $\boldsymbol{U} \colon \varTheta \to \mathbf{R}^Q$, which satisfies
  \begin{equation}
  \label{eq:Utheta}
  \boldsymbol{U}(\boldsymbol{\vartheta}) =
  \begin{bmatrix}
  u_{M,N}(\boldsymbol{x}^{(1)}, t^{(1)}, \boldsymbol{\vartheta}) \\
  \vdots \\
  u_{M,N}(\boldsymbol{x}^{(Q)}, t^{(Q)}, \boldsymbol{\vartheta})
  \end{bmatrix}.
  \end{equation}
Formally, the inverse problem can now be written as a parameter estimation problem
  \begin{equation}
  \label{eq:min}
  \argmin_{\boldsymbol{\vartheta} \in \varTheta} \lVert \boldsymbol{U} (\boldsymbol{\vartheta}) - \tilde{\boldsymbol{U}} \rVert_2^2,
  \end{equation}
where $\lVert \cdot \rVert_2$ denotes the Euclidean norm and the diffusivity $a$ can be computed from \eqref{eq:a}.
Due to the ill-posedness, however, the minimization has to be regularized in order to avoid meaningless reconstructions.

\subsection{Regularized nonlinear least squares}
\label{ssec:gn}
Let us briefly sketch a simple Gauss--Newton algorithm with line search for a minimization problem of the type \eqref{eq:min}.
Assume $\varTheta = \mathbf{R}^P$ for a moment.
Starting from $\boldsymbol{\vartheta}^{(0)} \in \varTheta$ and $k=1$, the algorithm produces a sequence of parameter vectors according to the following steps:
  \begin{enumerate}
  \item Solve a \emph{linear} least squares problem
    \begin{equation}
    \boldsymbol{\varDelta}_\text{opt} \coloneqq \argmin_{\boldsymbol{\varDelta} \in \mathbf{R}^P}
    \lVert \boldsymbol{J}_{\boldsymbol{U}}(\boldsymbol{\vartheta}^{(k-1)}) \boldsymbol{\varDelta} + \boldsymbol{U} (\boldsymbol{\vartheta}^{(k-1)}) - \tilde{\boldsymbol{U}} \rVert_2^2.
    \end{equation}
  \item Solve a one-dimensional optimization problem
    \begin{equation}
    \alpha_\text{opt} \coloneqq \argmin_{\alpha \in \mathbf{R}_+} \lVert \boldsymbol{U} (\boldsymbol{\vartheta}^{(k-1)} + \alpha \boldsymbol{\varDelta}_\text{opt}) - \tilde{\boldsymbol{U}} \rVert_2^2.
    \end{equation}
  \item Set $\boldsymbol{\vartheta}^{(k)} = \boldsymbol{\vartheta}^{(k-1)} + \alpha_\text{opt} \boldsymbol{\varDelta}_\text{opt}$ and increase $k$ by one.
  \end{enumerate}
These steps are repeated until a suitable stopping criterion is satisfied.
Here, $\boldsymbol{J}_{\boldsymbol{U}} \colon \mathbf{R}^P \to \mathbf{R}^{Q \times P}$ is the Jacobian matrix of the mapping $\boldsymbol{U}$.
The linear least squares problem in the first step is usually solved by QR decomposing the Jacobian \cite{golub13,nocedal99}, although it is also possible to make the algorithm more efficient by employing the conjugate gradient method in case $P$ is large \cite{hanke97,langer10,langer07}. (As our spectral Galerkin method does not allow very large $P$, we exclude such considerations.)
Assuming that the Jacobian $\boldsymbol{J}_{\boldsymbol{U}}$ and the vector $\boldsymbol{U}$ have already been evaluated at $\boldsymbol{\vartheta}^{(k-1)}$, the computational complexity of the QR decomposition is $\mathcal{O}(QP^2)$.
After that, the search direction $\boldsymbol{\varDelta}_\text{opt}$ can be obtained easily with triangular substitution.
The line search of the second step requires only few evaluations of $\boldsymbol{U}$, since an approximate solution for $\alpha_\text{opt}$ is usually enough.

We refer to \cite{nocedal99} for more detailed discussion about nonlinear least squares algorithms.
If $\varTheta \neq \mathbf{R}^P$, we also need to specify and implement constraints, which we ignored above.
A common alternative to the standard Gauss--Newton algorithm (or the one with line search) is the Levenberg--Marquardt method, which essentially consists of the same subproblems as Gauss--Newton:\ evaluating $\boldsymbol{U}$ and its Jacobian $\boldsymbol{J}_{\boldsymbol{U}}$, and solving a linear least squares problem involving the Jacobian.

If the problem \eqref{eq:min} is replaced by a Tikhonov regularized version
  \begin{equation}
  \label{eq:minreg}
  \argmin_{\boldsymbol{\vartheta} \in \varTheta} \left\{ \lVert \boldsymbol{U} (\boldsymbol{\vartheta}) - \tilde{\boldsymbol{U}} \rVert_2^2 + \lambda^2 \lVert \boldsymbol{G}(\boldsymbol{\vartheta}) \rVert_2^2 \right\}
  = \argmin_{\boldsymbol{\vartheta} \in \varTheta} \left\lVert \begin{bmatrix} \boldsymbol{U} (\boldsymbol{\vartheta}) - \tilde{\boldsymbol{U}} \\ \lambda \boldsymbol{G}(\boldsymbol{\vartheta}) \end{bmatrix} \right\rVert_2^2,
  \end{equation}
the same principles still apply, assuming that the regularization function $\boldsymbol{G}$ is easy to evaluate and differentiate, e.g., $\boldsymbol{G}$ is a matrix.
A Bayesian interpretation for the inverse problem often results in a similar minimization problem as \eqref{eq:minreg}, if the maximum a posteriori estimate is sought.
It is also possible to employ more advanced methods such as the iteratively regularized Gauss--Newton method and Lepski{\u\i} balancing principle \cite{bauer09}, while still having essentially the same subproblems.

\subsection{Evaluating polynomials and derivatives}

The nonlinear mapping $\boldsymbol{U} \colon \varTheta \to \mathbf{R}^Q$ in \eqref{eq:Utheta} can be decomposed as
  \begin{equation}
  \label{eq:Udecomp}
  \boldsymbol{U}(\boldsymbol{\vartheta}) = \boldsymbol{V} \boldsymbol{\varphi}(\boldsymbol{\vartheta}),
  \end{equation}
where the matrix $\boldsymbol{V} \in \mathbf{R}^{Q \times N}$ satisfies
  \begin{equation}
  V_{q,j} = \sum_{i=1}^M \hat{u}_{i,j}(t_q) \phi_i(\boldsymbol{x}_q)
  \end{equation}
and the nonlinear part $\boldsymbol{\varphi} \colon \varTheta \to \mathbf{R}^N$ is defined in \eqref{eq:varphij}.
The matrix $\boldsymbol{V}$ can be constructed in advance, that is, before performing any measurements or optimization.
On the other hand, the multivariate polynomials can be evaluated according to
  \begin{equation}
  \varphi_j(\boldsymbol{\vartheta}) = (\bar{\varphi}_0)^{P-\lvert \mathcal{L}_j \rvert} \prod_{p \in \mathcal{L}_j} \bar{\varphi}_{\varLambda_{j,p}}(\vartheta_p),
  \end{equation}
where 
  \begin{equation}
  \mathcal{L}_j \coloneqq \{ p \mid \varLambda_{j,p} \neq 0 \}
  \end{equation}
contains the indices to non-constant univariate polynomials.
For a total degree space, the number of such indices must be $\lvert \mathcal{L}_j \rvert \leq n$ for each $j = 1,\ldots,N$.
In this case, and assuming $n = \mathcal{O}(1)$, the evaluation of $\boldsymbol{\varphi}$ requires $\mathcal{O}(N)$ floating point operations.
The total complexity of evaluating $\boldsymbol{U}$ is thus determined by the matrix-vector multiplication \eqref{eq:Udecomp}, which takes $\mathcal{O}(NQ)$.

The Jacobian matrix $\boldsymbol{J}_{\boldsymbol{U}} \colon \varTheta \to \mathbf{R}^{Q \times P}$ can be written as
  \begin{equation}
  \label{eq:Jdecomp}
  \boldsymbol{J}_{\boldsymbol{U}}(\boldsymbol{\vartheta}) = \boldsymbol{V} \boldsymbol{J}_{\boldsymbol{\varphi}}(\boldsymbol{\vartheta}),
  \end{equation}
where the \emph{basis Jacobian} $\boldsymbol{J}_{\boldsymbol{\varphi}} \colon \varTheta \to \mathbf{R}^{N \times P}$ contains the partial derivatives of the multivariate polynomials.
Note that for a general argument $\boldsymbol{\vartheta} \in \varTheta$, the basis Jacobian has exactly the same sparsity structure as the degree matrix $\boldsymbol{\varLambda}$ defining the underlying polynomial space.
For that reason, the evaluation of $\boldsymbol{J}_{\boldsymbol{\varphi}}$ becomes quite cheap.
Indeed, for any $1 \leq j \leq N$ and $p \in \mathcal{L}_j$, we have
  \begin{equation}
  J_{\boldsymbol{\varphi}}(\boldsymbol{\vartheta})_{j,p} = (\bar{\varphi}_0)^{P-\lvert \mathcal{L}_j \rvert} \bar{\varphi}'_{\varLambda_{j,p}}(\vartheta_p)
  \prod_{q \in \mathcal{L}_j \setminus p} \bar{\varphi}_{\varLambda_{j,q}}(\vartheta_q),
  \end{equation}
and for $p \notin \mathcal{L}_j$, the corresponding entries in the basis Jacobian $\boldsymbol{J}_{\boldsymbol{\varphi}}$ vanish.
Again, if $\lvert \mathcal{L}_j \rvert = \mathcal{O}(1)$ for each $j$, the evaluation of $\boldsymbol{J}_{\boldsymbol{\varphi}}$ requires $\mathcal{O}(N)$ operations.
Due to the sparsity the matrix-matrix product in \eqref{eq:Jdecomp} takes only $\mathcal{O}(NQ)$.

From theorem \ref{thm:comb} we immediately get the following:
\begin{cor}
\label{cor:complexity}
The computational cost of evaluating the parametric solution $\boldsymbol{U}$ and all its first-order partial derivatives is $\mathcal{O}(P^n Q)$, if the polynomial space of total degree $n$ is used to discretize the parameter domain.
In particular, choosing $n=2$ yields the same complexity as the QR decomposition in a Gauss--Newton step.
\end{cor}

In other words, a quadratic approximation of the parameter dependence does not pose a bottleneck for the efficiency of a Gauss--Newton or similar minimization scheme.
The result is nontrivial in the sense that a naïve finite difference approximation of the Jacobian $\boldsymbol{J}_{\boldsymbol{U}}$ would require $P$ function evaluations and thus have a complexity of $\mathcal{O}(P^{n+1}Q)$.
Let us also emphasize at this point that the complexity of the inverse problem is completely independent of spatial and temporal discretizations of the forward problem and it is also irrelevant whether the spectral approximation was obtained with the Galerkin or collocation method.

The linearized subproblems in a Gauss--Newton scheme can also be solved by, e.g., a conjugate gradient iteration instead of the QR decomposition.
For general complexity analysis of such nested iterations, see \cite{langer10}.
Extending the analysis to include a polynomial surrogate model is beyond the scope of this article.

Notice that the forward parametric solution can be solved with a larger spectral basis than what is used when evaluating the solution and its Jacobian.
This is useful, since it is difficult to choose in advance the optimal subset of multivariate polynomials.
After having the parametric solution at hand, one can simply discard those polynomials that correspond to columns of the matrix $\boldsymbol{V}$ having smallest (Euclidean) norm.
It is also possible to use different subsets of polynomials for $\boldsymbol{\varphi}$ and for $\boldsymbol{J}_{\boldsymbol{\varphi}}$.
Having fewer polynomials for the Jacobian results in the so called perturbed Gauss--Newton method \cite{gratton07}.
In particular, it is even possible to solve a single regular forward problem at each step while approximating the Jacobian with the pre-computed polynomials.

Finally, note that it is also possible to treat the inverse problem that is based on several different intitial conditions, boundary fluxes or other components of the problem setting.
This merely requires stacking the corresponding matrices $\boldsymbol{V}$ on top of each other.

\section{Numerical examples}
\label{sec:num}

The proposed method is demonstrated with simulated boundary data.
First, we solve a two-dimensional parametric forward problem corresponding to $P = 14^2 = 196$ bi-quadratic B-splines that are uniformly spaced so that they form a partition of unity on the unit square $\varOmega = (0,1)^2$.
We choose the parameter interval $E = (1/2,2)$, constant weight $w$ and employ polynomial space of total degree $n=2$, resulting to $N = 19503$ multivariate Legendre polynomials in accordance with part (a) of theorem \ref{thm:comb}.
We assume zero initial condition $u_0 \equiv 0$ and also set $f \equiv 0$.
For the horizontal boundaries, we assume homogeneous Neumann conditions, i.e., $g|_{x_2 \in \{0,1\}} \equiv 0$, whereas the vertical boundaries satisfy $g|_{x_1 = 0} = -20 t$ and $g|_{x_1 = 1} = 20 t$.
This corresponds to the case where two sides of a square-shaped object are insulated and two sides are heated or cooled with a heat flux which is linear in time.
The spatial discretization is performed with $M = 37^2 = 1369$ uniformly spaced piecewise linear FEM basis functions (corresponding to $2592$ triangular elements) and for the semi-implicit Euler method we choose the time step $\delta = 10^{-3}$.

The boundary data is generated by using $\tilde{M} = 129^2 = 16641$ FEM basis functions and the second order Crank--Nicolson time integration method with a step length $\tilde{\delta} = 10^{-3}$.
The measurement consists of $Q_s = 36$ spatial points that are uniformly distributed across the boundary (including corners), and of $Q_t = 13$ time instances $t^{(q)} \in \{0.01, 0.05, 0.09, \ldots, 0.49\}$.
Thus, the measurement vector $\tilde{\boldsymbol{U}}$ has $Q = Q_s Q_t = 468$ elements.
For each value, we add independent zero mean Gaussian noise realizations with standard deviation $\sigma = \sigma_0 \cdot \max_{1 \leq j \leq Q} \tilde{U}_j$ with some $\sigma_0 > 0$.

The reconstructions are computed by minimizing \eqref{eq:minreg} with the \texttt{lsqnonlin} function of Matlab.
The default algorithm \texttt{trust-region-reflective} handles bound constraints that are chosen to agree with the parameter domain $\varTheta = E^P$.
The regularization function $\boldsymbol{G}$ is chosen to be a discretized Laplace operator (i.e., a matrix which in one dimension would be tridiagonal with values $-1$, $2$ and $-1$) so that the minimization prefers smooth diffusivities.
The regularization parameter $\lambda$ is set such that the Morozov discrepancy principle
  \begin{equation}
  \label{eq:morozov}
  \lVert \boldsymbol{U}(\boldsymbol{\vartheta}) - \tilde{\boldsymbol{U}} \rVert_2 \approx \sqrt{Q} \sigma
  \end{equation}
holds for the minimizing vector $\boldsymbol{\vartheta}$.
Because the minimization process is very fast, the adjustment of $\lambda$ can be done, for example, by trial-and-error.

In the first two-dimensional example we consider the boundary data corresponding to a smooth diffusivity
  \begin{equation}
  \label{eq:a1}
  \tilde{a}(\boldsymbol{x}) = 1.25 + \frac{\sin(6 x_1) \cos(4 x_2)}{2},
  \end{equation}
which satisfies $0.75 \leq \tilde{a}(\boldsymbol{x}) \leq 1.75$ for all $\boldsymbol{x} \in \varOmega$.
Figure \ref{fig:2d2} shows the target diffusivity together with the reconstruction.
Here, the noise parameter is $\sigma_0 = 0.001$ and for the regularization we use $\lambda = 0.025$.
We see that the quality of the reconstruction is good.
A piecewise constant diffusivity taking values $1/2$ and $3/2$ is reconstructed in figure \ref{fig:2d7} by using the aforementioned values for the noise and regularization.
As expected, the reconstruction is far from exact in this nonsmooth case, because the regularization favors smooth diffusivities and such a piecewise constant diffusivity is also impossible to represent by using bi-quadratic B-splines.

Let us also consider the approximation error
  \begin{equation}
  \label{eq:approxerror}
  \lVert \boldsymbol{U}(\boldsymbol{\vartheta}) - \tilde{\boldsymbol{U}}_{\boldsymbol{\vartheta}} \rVert_2,
  \end{equation}
where $\boldsymbol{\vartheta}$ is the minimizing vector as in \eqref{eq:morozov} and $\boldsymbol{U}_{\boldsymbol{\vartheta}}$ denotes the simulated high-accuracy noiseless ($\sigma_0 = 0$) boundary data that is obtained by choosing the coefficient $\tilde{a}$ to be the just reconstructed diffusivity.
This error results from truncating the polynomial expansion (i.e., having a finite $N$) and also from the less accurate spatial and temporal discretizations of the parametric solution.
In the examples shown in figures \ref{fig:2d2} and \ref{fig:2d7}, the error values \eqref{eq:approxerror} are $0.11$ and $0.10$, respectively, which in these low-noise cases are slightly larger than $\sqrt{Q} \sigma$ in \eqref{eq:morozov}.
However, although this approximation error is not accounted for in \eqref{eq:morozov}, no instability of the reconstructions was observed in any of the numerical tests.

Some other smooth target diffusivities and their reconstructions are shown in figures \ref{fig:2d5} and \ref{fig:2d6}.
Also in these cases the target diffusivity values lie between $1/2$ and $2$.
Now the noise parameter is considerably larger, namely $\sigma_0 = 0.02$.
In order to approximately satisfy the Morozov discrepancy principle \eqref{eq:morozov}, we set $\lambda = 0.4$.
The reconstructions are still qualitatively correct.
The approximation errors \eqref{eq:approxerror} related to figures \ref{fig:2d5} and \ref{fig:2d6} are $1.20$ and $0.13$, respectively, which are now smaller than the noise level in \eqref{eq:morozov}.
The larger error in the third example may be explained by the target diffusivity values that are quite low; the parametric solution is usually most accurate when the diffusivity values are in the middle of the parameter interval $E$.

  \begin{figure}[t]
  \centering
  \includegraphics{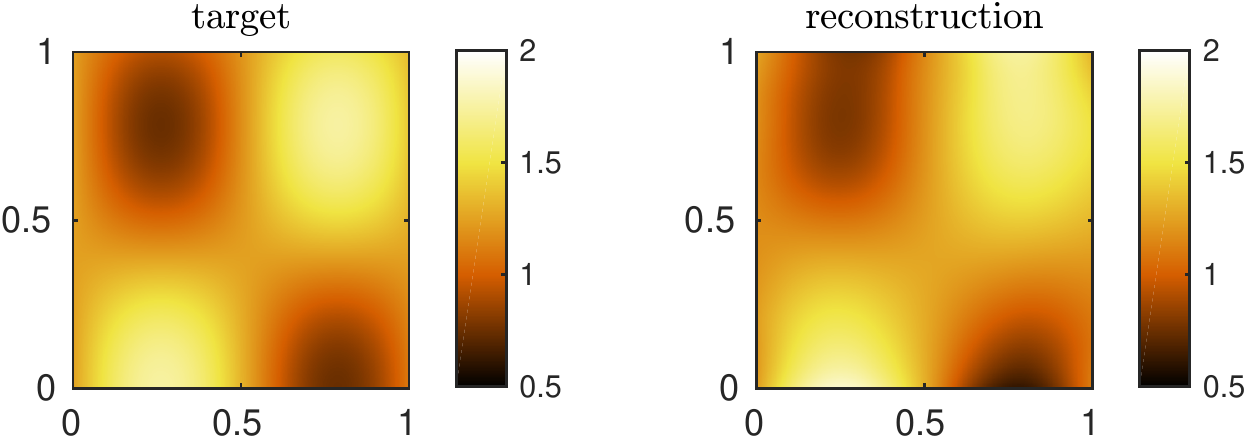}
  \caption{Two-dimensional target diffusivity $\tilde{a}$ \protect\eqref{eq:a1} on the left and the reconstruction $a$ on the right ($\sigma_0 = 0.001$).}
  \label{fig:2d2}
  \end{figure}

  \begin{figure}[t]
  \centering
  \includegraphics{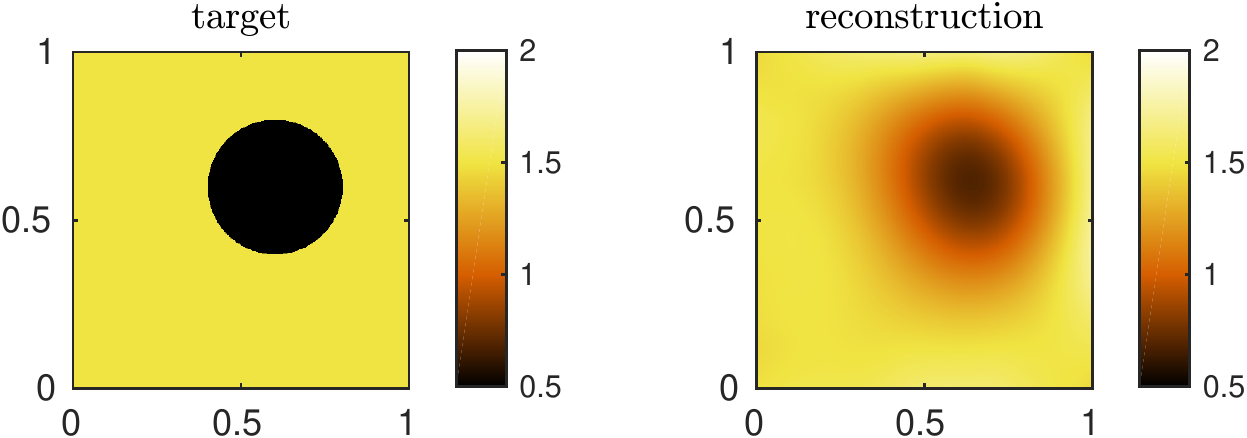}
  \caption{Two-dimensional piecewise constant target diffusivity $\tilde{a}$ on the left and the reconstruction $a$ on the right ($\sigma_0 = 0.001$).}
  \label{fig:2d7}
  \end{figure}

  \begin{figure}[t]
  \centering
  \includegraphics{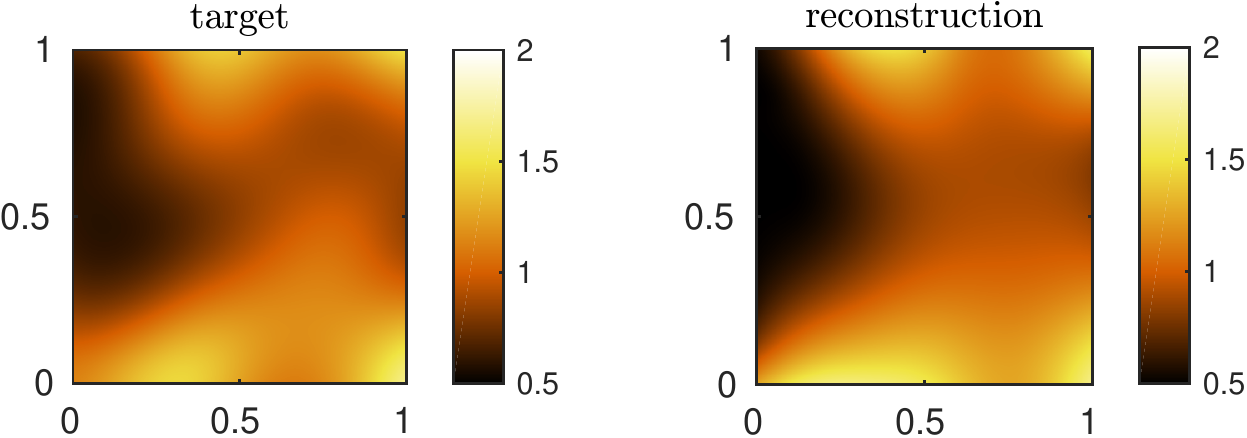}
  \caption{Two-dimensional target diffusivity $\tilde{a}$ on the left and the reconstruction $a$ on the right ($\sigma_0 = 0.02$).}
  \label{fig:2d5}
  \end{figure}

  \begin{figure}[t]
  \centering
  \includegraphics{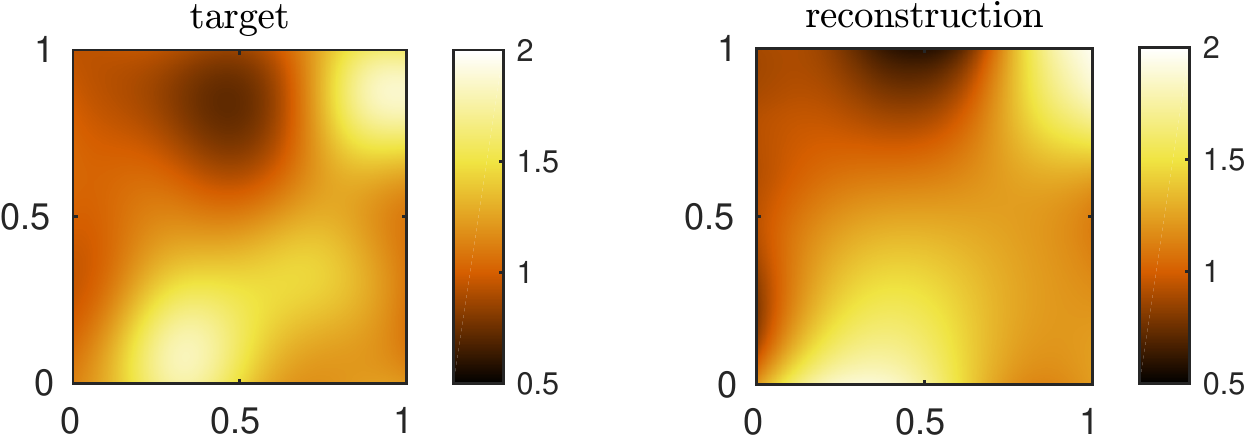}
  \caption{Two-dimensional target diffusivity $\tilde{a}$ on the left and the reconstruction $a$ on the right ($\sigma_0 = 0.02$).}
  \label{fig:2d6}
  \end{figure}

In our three-dimensional example we have $P = 6^3 = 216$ trilinear B-splines that form the partition of unity in the unit cube $\varOmega = (0,1)^3$.
As in the two-dimensional case, we use $E = (1/2,2)$, constant weight and $n=2$, which now results in $N = 23653$.
The Neumann boundary conditions on two opposing sides are $g|_{x_1 = 0} = -40 t$ and $g|_{x_1 = 1} = 40 t$, whereas the remaining four faces have homogeneous Neumann boundary conditions.
The initial value and the forcing term are set to zero as in the two-dimensional case.
In three-dimensional case, we use a piecewise linear finite element mesh with $M = 26^3 = 17576$ nodes and $93750$ tetrahedra.
The time step of the semi-implicit Euler method is still $\delta = 10^{-3}$.

The three-dimensional measurement data is generated by using $\tilde{M} = 65^3 = 274625$ piecewise linear basis functions, Crank--Nicolson time integration with a time step of $\tilde{\delta} = 10^{-3}$ and a diffusivity
  \begin{equation}
  \label{eq:a3}
  \tilde{a}(\boldsymbol{x}) = 1.25 + (0.5-x_3) \sin(6 x_1) \cos(4 x_2).
  \end{equation}
The measurement consists of $Q_s = 152$ spatial locations that are uniformly spaced across the boundary of the cube, and of $Q_t = 13$ time instances as before, so that $Q = Q_s Q_t = 1976$.
The noise model is the same as in the two-dimensional case with $\sigma_0 = 0.01$ and for the regularization we choose $\lambda = 0.09$.
Figure~\ref{fig:3d2} shows that the reconstruction is qualitatively correct in this three-dimensional setting.

  \begin{figure}[ht]
  \centering
  \includegraphics[scale=.9]{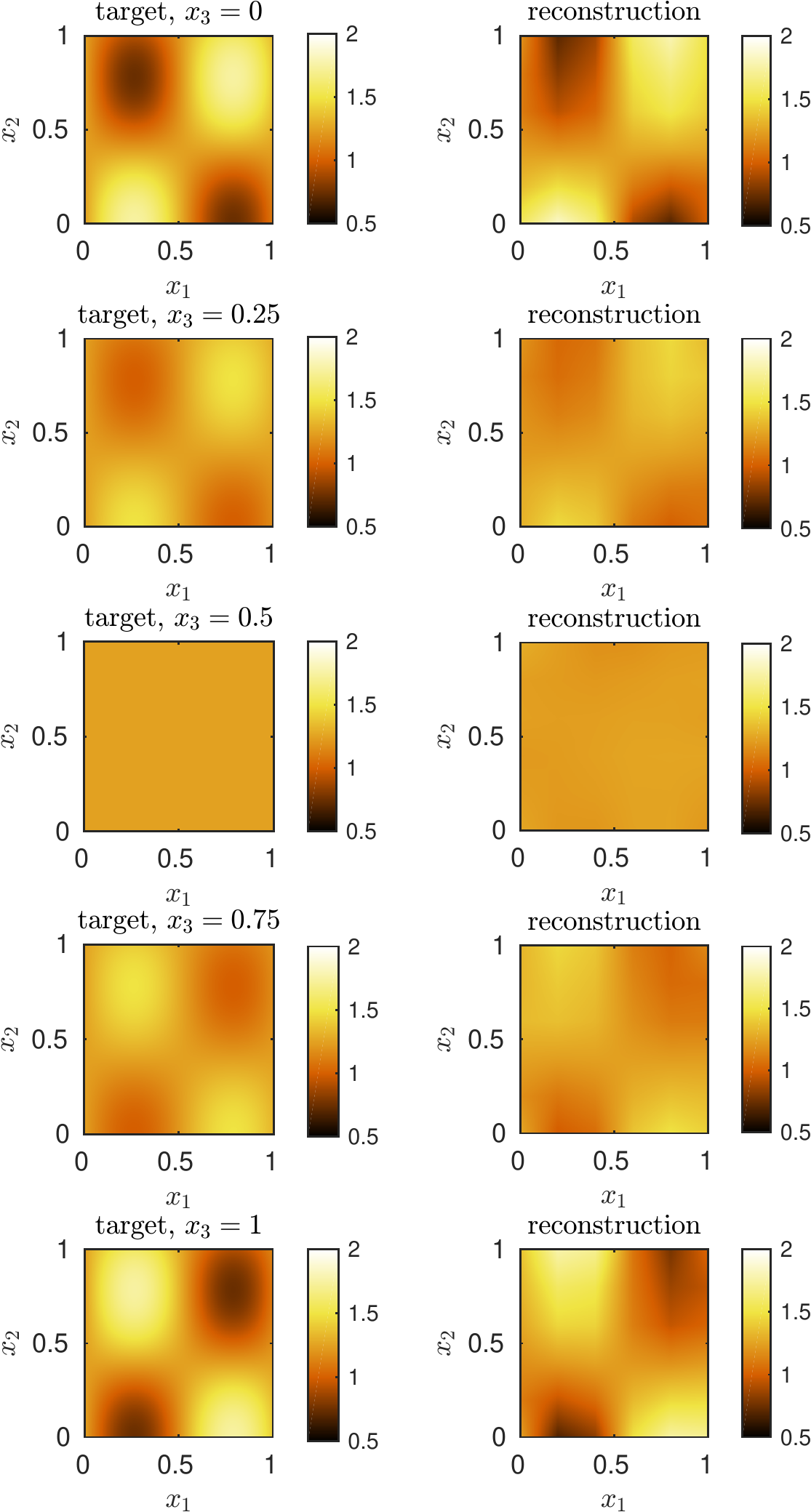}
  \caption{Left:\ slices of the three-dimensional target diffusivity \protect\eqref{eq:a3} in the unit cube. Right:\ corresponding slices of the reconstruction ($\sigma_0 = 0.01$).}
  \label{fig:3d2}
  \end{figure}

All computations were performed by using Matlab R2014b.
Computationally the most demanding case is the three-dimensional parametric forward problem in which $MN \approx 4 \cdot 10^8$ and quite a large amount of memory is needed even to store the intermediate vector during the time integration.
On the other hand, the measurement consists of only $Q \approx M/10$ values and thus the matrix $\boldsymbol{V} \in \mathbf{R}^{Q \times N}$ can be handled easily even in the three-dimensional case.
For a fixed regularization parameter, the reconstruction itself took at most a few seconds in all numerical experiments on an up-to-date desktop computer.
This is the only step that cannot be performed offline, that is, prior to the measurements.
Further improvements could still be expected by implementing a customized least squares solver and fine-tuning the stopping criterion for the iterative optimization.
Moreover, combining our method with a state-of-the-art iteratively regularized Gauss--Newton method with inner conjugate gradient iteration is left for future studies.

\section{Conclusions}
\label{sec:conclusions}

We have studied a time-dependent parametric partial differential equation and the related inverse boundary value problem.
The parametric forward problem was solved by using the spectral Galerkin method in the parameter domain, finite element method in the spatial domain and a semi-implicit Euler method in the time interval.
The inverse problem was interpreted as a nonlinear least squares problem accompanied with a Tikhonov regularization.

It was shown that having a quadratic approximation for the parameter dependence results in an efficient inverse algorithm, where the evaluation of the multivariate polynomials and their derivatives does not constitute an essential bottleneck for a Gauss--Newtonian least squares method, if QR decompositions are used for solving the linear subproblems.
In particular, the inverse problem can be solved independently of the physical discretization of the forward problem.
The numerical results indicated that the quadratic approximation is indeed accurate enough for qualitatively correct reconstructions.
See \cite{langer10,langer07} for alternative approaches that are more efficient if the number of parameters is large and the QR decomposition becomes too expensive.

The proposed method is versatile and can be applied to elliptic problems, including EIT, as well.
The future research will concentrate on different regularization techniques such as sparsity promoting terms that do not satisfy the differentiability assumptions posed here.
In addition, different parametrizations will be studied.

\section*{Acknowledgments}
The author would like to thank professor Nuutti Hyvönen and the anonymous referees for comments and suggestions.

\bibliographystyle{ppde}
\bibliography{ref}

\end{document}